\documentclass[11pt,a4paper]{amsart}

\usepackage{setspace}

\usepackage{parskip}  

\usepackage{lmodern}
\usepackage[T1]{fontenc}
\usepackage[utf8]{inputenc}
\usepackage[english]{babel}
\usepackage[activate={true,nocompatibility},final,tracking=true,kerning=true,spacing=true,factor=1100,stretch=10,shrink=10]{microtype}
\microtypecontext{spacing=nonfrench}    

\usepackage{amsmath,amssymb,amsfonts,amsthm}    
\usepackage{mathtools}                          
\usepackage{bm}                                 
\usepackage{thmtools}                           
\usepackage{thm-restate}                        
\usepackage{esint}                              
\usepackage{mathdots}                           
\usepackage{dsfont}                             
\usepackage{diffcoeff}                          
\usepackage{nicematrix}                         
\usepackage{tikz-cd}                            
\usepackage{pgfplots}                           

\usepackage[marginparwidth=2.5cm]{geometry}     

\usepackage{enumitem}                           
\usepackage{xcolor}                             
\usepackage{comment}                            
\usepackage{graphicx}                           
\usepackage{nicefrac}                           

\pgfplotsset{compat=1.18}

\usepackage[colorlinks = true]{hyperref}        
\usepackage[nameinlink,capitalise,noabbrev]{cleveref}                           
\usepackage[textsize=scriptsize]{todonotes}       
\usepackage[backend=biber,giveninits=true,style=numeric,sorting=nyt,maxcitenames=3,date=long, doi = false, isbn = false, url = false, eprint = false]{biblatex}
\usepackage{csquotes}       

\def\NT@numberwithin{section}      

\usepackage{etoolbox}

\makeatletter
\renewenvironment{proof}[1][\proofname]{%
  \par\pushQED{\qed}%
  \normalfont
  \topsep6\p@\@plus6\p@\relax
  \trivlist
  \item[\hskip\labelsep\textsc{#1}\@addpunct{.}]%
}{%
  \popQED\endtrivlist\@endpefalse
}
\makeatother

\renewcommand{\proofname}{Proof}

\declaretheoremstyle[
  spaceabove=6pt, spacebelow=6pt,
  headfont=\bfseries,
  bodyfont=\itshape,
  headpunct=.,
  postheadspace=0.5em
]{thmplain}

\declaretheoremstyle[
  spaceabove=6pt, spacebelow=6pt,
  headfont=\bfseries,
  bodyfont=\normalfont,
  headpunct=.,
  postheadspace=0.5em
]{thmdef}

\declaretheoremstyle[
  spaceabove=6pt, spacebelow=6pt,
  headfont=\itshape,
  bodyfont=\normalfont,
  headpunct=.,
  postheadspace=0.5em
]{thmremark}

\declaretheorem[name=Theorem,style=thmplain,numberwithin=\NT@numberwithin]{theorem}

\declaretheorem[name=Proposition,style=thmplain,sibling=theorem]{proposition}

\declaretheorem[name=Theorem,style=thmplain,numbered=no]{theorem*}
\declaretheorem[name=Lemma,style=thmplain,numbered=no]{lemma*}
\declaretheorem[name=Proposition,style=thmplain,numbered=no]{proposition*}
\declaretheorem[name=Corollary,style=thmplain,numbered=no]{corollary*}
\declaretheorem[name=Definition,style=thmdef,numbered=no]{definition*}
\declaretheorem[name=Remark,style=thmremark,numbered=no]{remark*}
\declaretheorem[name=Example,style=thmremark,numbered=no]{example*}

\crefname{theorem}{theorem}{theorems}
\Crefname{theorem}{Theorem}{Theorems}
\crefname{lemma}{lemma}{lemmas}
\Crefname{lemma}{Lemma}{Lemmas}
\crefname{proposition}{proposition}{propositions}
\Crefname{proposition}{Proposition}{Propositions}
\crefname{corollary}{corollary}{corollaries}
\Crefname{corollary}{Corollary}{Corollaries}
\crefname{definition}{definition}{definitions}
\Crefname{definition}{Definition}{Definitions}
\crefname{remark}{remark}{remarks}
\Crefname{remark}{Remark}{Remarks}
\crefname{example}{example}{examples}
\Crefname{example}{Example}{Examples}

\newtheorem*{claim*}{Claim} 

    \newcommand{\1}[0]{\mathds{1}}

    \newcommand{\Z}[0]{\mathbb{Z}}

    \newcommand{\E}[0]{\mathbb{E}}
    \renewcommand{\P}[0]{\mathbb{P}}

    \def\d{\delta}
    \newcommand{\e}[0]{\varepsilon}

    \def\Var{{\text{Var}}}

    \DeclarePairedDelimiterX\AP[2]{(}{)}{#1 \,\bmod\, #2}
    \DeclareMathOperator*{\argmin}{arg\,min}

\def\ts{{\nu}^{(c)}}
\def\tn{{\nu}^{(d)}}

\title{On low-discrepancy sequences and Poissonian pair correlation}
\author{Hannah Porath}
\date{}

\begin{document}

\address{Institute of Analysis and Number Theory, TU Graz, Austria}
\email{hannah.porath@tugraz.at}
\begin{abstract}
Uniform distribution modulo 1 is a classical notion of pseudo-randomness for sequences in the unit interval, which is quantified in terms of the discrepancy. Sequences whose discrepancy is of the smallest possible asymptotic order are called low-discrepancy sequences. The Poissonian pair correlation is another notion of pseudo-randomness, which studies the distribution of the gaps between pairs of elements of the sequence on a local scale. It is known that Poissonian pair correlation implies uniform distribution mod 1, and that the opposite implication is not true in general. It has also been observed that classical examples of low-discrepancy sequences fail to have Poissonian pair correlation, and it has been speculated that the two properties might be irreconcilable due to the high degree of structural rigidity that is required for low-discrepancy behavior. As we prove in this paper, this is not the case: we construct an example of a low-discrepancy sequence with Poissonian pair correlation. 
\end{abstract}

\maketitle

\section{Introduction}

A classical problem in mathematics is to characterize the ``pseudo-random'' nature of sequences. For real numbers in the unit interval, a classical notion of pseudo-randomness is uniform distribution modulo 1 (also called equidistribution). The paper from 1916 of Hermann Weyl \cite{weyl} is widely attributed with establishing the theory of uniform distribution modulo 1 (u.d.\ mod 1) as a full-fledged mathematical discipline. A sequence $(x_n)_{n \geq 1}$ in $[0,1)$ is said to be uniformly distributed modulo 1 if 
\begin{equation} \label{weyl_crit}
\lim_{N \to \infty} \frac{1}{N} \sum_{n=1}^N \1_{[a,b)}(x_n) = b-a
\end{equation}
for all intervals $[a,b) \subset [0,1)$. 
To formulate a quantitative version of \eqref{weyl_crit}, we need the notion of discrepancy of a sequence: For a sequence $(x_n)_{n \geq 1} \subset [0,1),$ the discrepancy (more precisely: the discrepancy of its finite truncation to the first $N$ initial points) is defined as
\begin{equation*}
    D_N(x_1,\dotsc, x_N) = \sup_{0 \leq a < b < 1} \left|\frac{1}{N} \sum_{1 \leq n \leq N} \1_{[a,b)}(x_n) - (b-a)\right|.
\end{equation*}
Similarly, the star-discrepancy of this sequence is defined as
$$
D_N^*(x_1, \dots, x_N) = \sup_{a \in [0,1)} \left| \frac{1}{N} \sum_{n=1}^N \1_{[0,a)}(x_n) -  a \right|.
$$
Those notions are closely related via
\begin{equation*}
    D_N^* \leq D_N \leq 2 D_N^*.
\end{equation*}

It is easily seen that the star-discrepancy of a finite point set $\{x_1, \dots, x_N\}$ cannot be smaller than $\frac{1}{2N}$, which is realized for the equally-spaced set $\{\frac{2k-1}{2N},\, 1 \leq k \leq N \}$. A fundamental fact, which was first established by van Aardenne-Ehrenfest \cite{vAE}, is that such a small asymptotic order of the star-discrepancy cannot be achieved for all subsets $\{x_1, \dots, x_N\}$ of an infinite sequence $(x_n)_{n \geq 1}$. The final quantitative form of this fact was established by Schmidt \cite{schmidt}, who proved that there is $c_\textup{abs}>0$ such that for every sequence $(x_n)_{n \geq 1}$,
$$
D_N^* (x_1, \dots, x_N) \geq c_{\textup{abs}} \frac{\log N}{N} \qquad \text{for infinitely many $N$.}
$$
This is optimal, since there are examples of infinite sequences whose star-discrepancy is of asymptotic order $O \left( \frac{\log N}{N} \right)$. Sequences which achieve this star-discrepancy bound are called low-discrepancy sequences. One classical example of such a sequence is the sequence of fractional parts $(\{n \alpha\})_{n \geq 1}$ (Kronecker sequence) for an $\alpha$ which is badly approximable in the sense of Diophantine approximation (see for example \cite[p.\ 125]{kn}). Another classical example of a low-discrepancy sequence is the van der Corput sequence, which is the starting point of the proof of the main theorem of the present paper and will be introduced in detail in the next section. There are further constructions of low-discrepancy sequences, based for example on an ergodic theory perspective \cite{carbone,nino} or on greedy-type algorithms \cite{kritz, paus, stein}. We have motivated the notion of uniform distribution modulo 1 from the desire to characterize pseudo-random behavior of real sequences. 
We note that uniform distribution modulo 1 is a pseudo-randomness property in the sense that almost all realizations of a sequence of independent, identically distributed (i.i.d.) random variables having uniform distribution on $[0,1)$ are uniformly distributed modulo 1; however, in the same sense being low-discrepancy is \emph{not} a pseudo-randomness property, since a typical realization of a random sequence will have a star-discrepancy  of order roughly $N^{-1/2}$ rather than $\frac{\log N}{N}$.
 
In recent years, ``local'' test statistics for pseudo-random behavior have been a major focus of mathematical research. These statistics are inspired by problems in theoretical physics, such as the Berry-Tabor conjecture (see \cite{mark} for an exposition). From the perspective of theoretical physics, a key purpose of these local statistics is to detect ``rigidity'' or ``clustering'' phenomena in the energy spectrum of quantum systems. One of the most important local test statistics, which is the one in the focus of the present paper, is the pair correlation. For a sequence $(x_n)_{n \geq 1}$ in $[0,1)$, set
\begin{equation*}
    R_{s,N}(x_1, \dots, x_N) = \frac{1}{N}\#\left\{1 \leq n \neq m \leq N \ | \ \|x_n-x_m\| \leq \frac{s}{N}\right\}.
\end{equation*}
Here $\| \cdot \|$ denotes the distance to the nearest integer. The sequence is said to exhibit Poissonian pair correlation if 
$$
R_{s,N} (x_1, \dots, x_N) \to 2s \qquad \text{as $N \to \infty$,}
$$
for every fixed $s>0$. This is a pseudo-randomness property: a realization of an i.i.d.\ sequence of random points (uniformly distributed in $[0,1)$) has Poissonian pair correlation almost surely. Establishing Poissonian pair correlation for a sequence of arithmetic origin is typically very difficult, but a few examples are known \cite{ebmv,hauke2026badroughrotationpoissonian,lst,lt}. It is known that Poissonian pair correlation implies uniform distribution modulo 1 \cite{alp,gl,Hauke_Zafeiropoulos_2023,markpair}. The converse is not true; for example, the Kronecker sequence $(\{n \alpha\})_{n \geq 1}$ is u.d.\ mod 1 for irrational $\alpha$, but does not have Poissonian pair correlation for any $\alpha$. It is remarkable that those sequences which have been identified to possess Poissonian pair correlation usually seem to have ``large'' star-discrepancy (of order no less than $N^{-1/2}$). On a heuristic level, this might just be an consequence of the fact that Poissonian pair correlation is a generic property of a random sequence, while low discrepancy is not, so that among sequences of small discrepancy it is ``unlikely'' to find a specimen with Poissonian pair correlation, and vice versa. Several constructions of low-discrepancy sequences have been tested for Poissonian pair correlation, and so far the outcome has always been negative. Much of the work in this direction has been initiated by Larcher and Stockinger \cite{ls1,ls2,ls3}. As already noted, the Kronecker sequence $(\{n \alpha\})_{n \geq 1}$ has low discrepancy for badly approximable $\alpha$, but does not have Poissonian pair correlation for any $\alpha$. This particular rigidity in the gap structure of the Kronecker sequence is most famously examplified in the Three Gap Theorem of S{\'o}s \cite{sos} and \'{S}wierczkowski \cite{swi}. In \cite{ls2}, Larcher and Stockinger placed this example in a general framework, and showed that for any sequence with the ``finite gap property'', the pair correlation cannot be Poissonian (see also \cite{allm,fw}). Overall, it is tempting to get the impression that there is a fundamental dichotomy at work here: low-discrepancy sequences seem to require a strong rigid structure in order to achieve the low-discrepancy property -- but such a rigid behavior is irreconcilable with pseudo-random behavior on a local level, i.e. Poissonian pair correlation. Questions about the relation of low-discrepancy behavior and Poissonian pair correlation have been addressed in the recent literature in several instances, see for example \cite{becher,hklsu,sw,ws}. Manuel Hauke-Treuer asked at the MCQMC conference in Linz 2022 the specific question whether a low-discrepancy sequence can have Poissonian pair correlation. The purpose of the present paper is to give an affirmative answer to this question.

\begin{theorem}\label{th_1}
There exist sequences $(x_n)_{n \geq 1}$ in $[0,1)$ whose discrepancy is of order 
$$
D_N^*(x_1, \dots, x_N) = O \left( \frac{\log N}{N} \right), 
$$
and which have Poissonian pair correlation.
\end{theorem}

Our existence proof employs a randomized variant of the van der Corput sequence. We will present the construction of our example in the next section.

\section{Construction}

The binary van der Corput sequence is one of the most classical examples of a low-discrepancy sequence, and it is constructed as follows: Writing the positive integer $n$ in its binary representation $n = \sum_{i \geq 0} b_i2^i,$ we set
$$x_n = \sum_{i \geq 0}b_i2^{-i-1},$$
which yields the sequence $(x_n)_{n \geq 1} = \big(\frac{1}{2},\frac{1}{4},\frac{3}{4}, \frac{1}{8}, \cdots\big).$ The discrepancy of this sequence is known to satisfy
$$
D_N^*(x_1, \dots, x_N) \leq \frac{\log (N+1)}{N \log 2},
$$
see \cite[p. 127]{kn}. It is also known that this sequence does not have Poissonian pair correlation; see \cite{weiss} for a detailed calculation. When aiming to create a low-discrepancy sequence with Poissonian pair correlation, our approach is to randomly perturb a deterministic low-discrepancy sequence (such as the van der Corput sequence) with random fluctuations on a local level, i.e. using random $y_n \sim \text{Unif}([-g(n),g(n)])$ and then setting 
$$\ts_n = x_n + y_n \mod 1, \qquad n \geq 1.$$
We will show in Proposition \ref{ppccont} below that such a randomly perturbed sequence does indeed exhibit Poissonian pair correlation for a typical realization, if the perturbation range $g(n)$ is suitably chosen as a function of $n$ (not too small, so that the random perturbation does indeed lead to Poissonian pair correlation -- essentially, the suitable condition for this is that $g(n) = \omega(1/n)$, i.e., $n g(n) \to \infty$ as $n \to \infty$). However, it turns out to be rather difficult to preserve the low-discrepancy property. This is a bit surprising, since the low-discrepancy requirement of discrepancy of order $\frac{\log N}{N}$ seems to leave ample room for the $\omega(1/n)$-size random perturbations. However, the crucial fact is that a point $x_n$ has to be perturbed with a distance $g(n)$ that depends on its individual index $n$, while the discrepancy is calculated with respect to the index $N$ of the whole initial point set of points with indices $\{1,\dots, N\}$. In other words, a ``small'' perturbation of a point with a small index $n$ can be very large relative to the scaling variable $N$ which is later used to calculate the discrepancy of a point set with indices in $\{1, \dots, N\}$. This implies that, while it is not so difficult to randomize a triangular array $\{x_1, \dots, x_N\}$ (for any given $N$), in such a way that the randomized version essentially has Poissonian pair correlation, it is very delicate to apply the same procedure to an infinite \emph{sequence} of points, since any (small) changes made early in the constructed sequence will permanently remain as part of the sequence, and will appear to ``blow up'' relative to the precision level required for checking low-discrepancy behavior as the total number of points advances. 

To remedy this we instead consider a discretized random perturbation of the van der Corput sequence, and show in Proposition \ref{samepc} that this yields the same asymptotic pair correlation statistic as the sequence $(\ts_n)_{n \geq 1},$ which arises from a continuous random perturbation. Namely, we randomly move a van der Corput point $x_n = \frac{2k+1}{2^i}$ to a randomized position on a finer dyadic scale. The fact that this randomly perturbed point can also be represented as a dyadic fraction means that it is sooner or later bound to appear as an element of the (unperturbed) van der Corput sequence, and we can identify the perturbed point (with the smaller index) with the van der Corput point with the larger index. We then move this van der Corput point our sequence point is identified with to the original place. In this way we can ``recycle'' the perturbed points from earlier stages in the construction as points that are supposed to show up in the van der Corput sequence at a later stage. This construction solves the problem of the perceived ``blow-up'' of the size of the perturbation of our points, which we discussed above: Rather than keeping a perturbed point in perpetuity as a source of disturbance for the star-discrepancy calculation, we ``heal'' the disruption caused by the perturbed point by incorporating it into the construction of the van der Corput sequence (where all dyadic fractions, including the ones obtained in our perturbation procedure, have to show up sooner or later). This is the rough outline of the construction, but it is clear that the actual construction will require to choose all the relevant parameters in a carefully balanced way.

\textsc{Notation.} Throughout, let $0 < \e < \frac{1}{2}$ be fixed. We denote by $B_i$ the dyadic block $[2^{i},2^{i+1}) \cap \Z.$  Further, let $i_n = \lfloor \log_2 n\rfloor$ be such that $n \in B_{i_n}.$ Let $\|\cdot\|$ denote the distance to the nearest integer. \\
The symbols $O(\cdot)$, $o(\cdot)$, $\ll$, $\gg$, and $\asymp$ are used with their usual meanings: We write $f = O(g)$ or  $f \ll g$ if there exists a constant $C>0$ such that $f \le Cg$, where the implied constant may depend on fixed parameters specified in the surrounding context but is otherwise absolute. The notation $f \asymp g$ means both $f \ll g$ and $f \gg g$. Finally, $f = o(g)$ means that $f/g \to 0$.

\medskip

Define the parameters
\begin{equation}\label{par}
    g(n) = \frac{\log\log i_n}{2^{i_n}}, \text{ and } j_n = i_n + \lfloor (1-\e)\log i_n\rfloor.
\end{equation}
Notice that $g(\cdot)$ is constant on each dyadic block $B_i.$ Write $\d_n = \lfloor (1-\e)\log i_n\rfloor,$ respectively $\d = \d(i) = \lfloor(1-\e)\log i\rfloor.$\\
We define the candidate set
\begin{equation*}
    \mathcal{C}_n = \left\{ \frac{c_n}{2^{j_n}} \ \Big| \ c_n \in [1,\dotsc, 2^{j_n}-1] \text{ odd}, \ \left|\frac{c_n}{2^{j_n}}-x_n\right| < g(2^{i_n}) \right\}.
\end{equation*}
This set contains van der Corput points with indices in $B_{j_n}$ in a $g(2^{i_n})$-neighborhood of the point $x_n$ we want to perturb. 

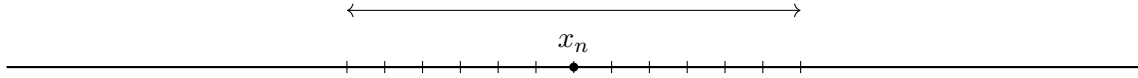
\begin{tikzpicture}[
    point/.style={circle, fill, inner sep=1.2pt},
    arr/.style={->, thick}
]
\begin{scope}
    \draw[thick](0,0)--(15,0);
    \node[point][label = {$x_n$}] (xn) at (7.5,0) {};
    \draw[<->] (7.5 -3.00,0.75) -- (7.5 + 3.00,0.75);
    \foreach \x in {0,...,12} {\draw[] ({7.5-3+0.5*\x},-0.08) -- ({7.5-3+0.5*\x},0.08);}
\end{scope}
\end{tikzpicture}

We randomly draw $y_n  \sim \text{Unif}([-g(n),g(n)])$ independently for each $n \geq 1$ and define the continuously perturbed sequence
\begin{equation*}
    \ts_n = x_n + y_n \mod 1.
\end{equation*}
In Proposition \ref{ppccont} we show that this sequence has Poissonian pair correlation almost surely. \\

We turn to defining a discrete version of the perturbation. Start with a realization of $(\ts_n)$ and put
$$\tn_n = \min \argmin_{\substack{\nu \in \mathcal{C}_n \\ (\nexists \ k < n) \ \tn_k = \nu}} |\nu-\ts_n|,$$
i.e. we pick the nearest unoccupied finer grid point as the point $(\tn_n),$ where, in case of a tie, we choose the leftmost point. Such a free candidate always exists: There are
\begin{equation*}
    \# \mathcal C_n \asymp i_n^{1-\e}\log\log i_n
\end{equation*}
candidates, and at most
\begin{equation*}
    \#\{k < n \ | \ \tn_k = \nu \in \mathcal C_n \} \ll 2^{i_n}g(2^{i_n}) = \log\log i_n
\end{equation*}
already occupied points.

To implement ``point recycling'', we then finally consider
\begin{equation*}
    \nu_n = \begin{cases}
            x_k, & \text{if $\tn_k = x_n$ for some $k < n$}, \\
            \tn_n, & \text{else}.
    \end{cases}
\end{equation*} 
This construction recycles previously perturbed points as follows.
If a point $x_k$ is moved to the points $\tn_k = x_n$ for some $n > k$ (meaning, a later van der Corput point), we, instead of perturbing $x_n$ as usual, use this point to restore the point $x_k,$ that is $\nu_n = x_k.$

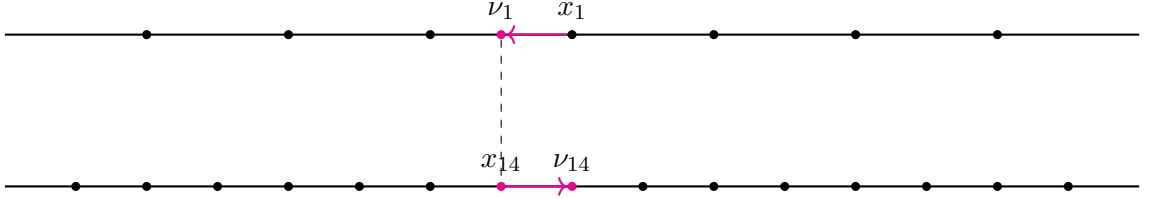
\def\L{15}
\begin{tikzpicture}[
    point/.style={circle, fill, inner sep=1.2pt},
    arr/.style={->, thick}
]

\begin{scope}
    \draw[thick] (0,0) -- (15,0) ;
    \node[point][label = {$x_1$}] (x1) at (0.5*\L,0) {} ;
    \node[point][label = {}] (x2) at (0.25*\L,0) {};
    \node[point][label = {}] (x3) at (0.75*\L,0) {} ;
    \node[point][label = {}] (x4) at (1/8*\L,0) {};
    \node[point][label = {}] (x5) at (5/8*\L,0) {} ;
    \node[point][label = {}] (x6) at (3/8*\L,0) {};
    \node[point][label = {}] (x7) at (7/8*\L,0) {} ;
    \node[point, magenta][label = {$\nu_1$}] (n1) at (7/16*\L,0) {} ;
\end{scope}

\begin{scope}[yshift= -2cm]
    \draw[thick] (0,0) -- (15,0) ;
    \node[point,magenta][label = {$\nu_{14}$}] (n14) at (0.5*\L,0) {} ;
    \node[point][label = {}] (x2) at (0.25*\L,0) {};
    \node[point][label = {}] (x3) at (0.75*\L,0) {} ;
    \node[point][label = {}] (x4) at (1/8*\L,0) {};
    \node[point][label = {}] (x5) at (5/8*\L,0) {} ;
    \node[point][label = {}] (x6) at (3/8*\L,0) {};
    \node[point][label = {}] (x7) at (7/8*\L,0) {} ;
    \node[point][label = {}] (x8) at (1/16*\L,0) {} ;
    \node[point][label = {}] (x9) at (9/16*\L,0) {};
    \node[point][label = {}] (x10) at (5/16*\L,0) {} ;
    \node[point][label = {}] (x11) at (13/16*\L,0) {};
    \node[point][label = {}] (x12) at (3/16*\L,0) {} ;
    \node[point][label = {}] (x13) at (11/16*\L,0) {};
    \node[point, magenta][label = {$x_{14}$}] (x14) at (7/16*\L,0) {} ;
    \node[point][label = {}] (x15) at (15/16*\L,0) {} ;
\end{scope}

\draw[dashed] (n1) -- (x14);
\draw[arr, magenta] (x14) -- (n14);
\draw[arr, magenta] (x1) -- (n1);

\end{tikzpicture}

\textit{Remark.} This sequence is essentially designed to have low discrepancy, it is constructed to mimic the van der Corput sequence globally. The recycling changes relatively few points per scale, which is used heavily in Proposition \ref{resto}, but ensures that, up to a few points, the point set arising from the construction is a reordering of the van der Corput point set $(x_n)_{n \leq N}.$ This will be made more apparent in Proposition \ref{disc}.

\medskip

\section{Discrepancy}

In the proof, we will work with the star-discrepancy defined as
\begin{equation*}
    D_N^*((\nu_n)_{n \leq N}) = \sup_{a \in [0,1)} \left|\frac{1}{N}\sum_{n \leq N}\1_{[0,a)}(\nu_n) - a\right|,
\end{equation*}
and we recall that $D_N^* \leq D_N \leq 2D_N^*.$

\begin{proposition}\label{disc}
    The sequence constructed above has low discrepancy, i.e.
    \begin{equation*}
        D_N\left((\nu_n)_{n \leq N}\right) = O\left(\frac{\log N}{N}\right).
    \end{equation*}
\end{proposition}

\textit{Remark.} This statement is \textit{deterministic} and holds for all realizations of $(\nu_n).$ The implied constant does not depend on the realization.

\begin{proof} To relate the discrepancy of $(\nu_n)_{n \geq 1}$ to the (low) discrepancy of the van der Corput sequence $(x_n)_{n \geq 1},$ we observe the following. A point $\nu_n$ is either identified as a van der Corput point $x_{m(n)}$ with $m(n) > n$ or a van der Corput point $x_{\ell(n)}$ with $\ell(n) < n.$ As long as $m (n) \leq N,$ the point $x_n$ will be restored as $\nu_{m(n)}.$ Define
\begin{equation*}
    \mathcal B(N) = \{1 \leq n \leq N \ | \ \nu_n = x_{m(n)} \text{ with } m(n) > N\},
\end{equation*}
the set of indices for which $x_n$ does not appear in $\{\nu_1,\dotsc,\nu_N\}.$ Since $n \leq N,$ $\ell(n) \leq N$ for all $n.$ \\
Record that
\begin{equation*}
    D_N^*((\nu_n)_{n \leq N}) \leq D_N^*((x_n)_{n \leq N}) + \frac{1}{N}\sup_{t \in [0,1)}\left|\#\{1 \leq n \leq N \ | \ \nu_n \leq t\} - \#\{1 \leq n \leq N \ | \ x_n \leq t\}\right|.
\end{equation*}
We want to show that the displacement error
\begin{equation}\label{discplace}
    \frac{1}{N}\sup_{t \in [0,1)}\left|\#\{1 \leq n \leq N \ | \ \nu_n \leq t\} - \#\{1 \leq n \leq N \ | \ x_n \leq t\}\right|
\end{equation}
is negligible against $\frac{\log N}{N}.$ Notice that, if $n \notin \mathcal B(N),$ we have a correspondence $x_n = \nu_{m(n)}$ and $x_{m(n)} = \nu_n,$ so these do not contribute to the displacement error, so, for $t \in [0,1)$ fixed,
\begin{align*}
    \#\{1 \leq n \leq N \ | \ \nu_n \leq t\} - \#\{1 \leq n \leq N \ | \ x_n \leq t\} &= \sum_{n \in \mathcal{B}(N)} \1(x_n \leq t) - \1(\nu_n \leq t) \\
    &=\sum_{n \in \mathcal B(N)} \1(x_n \leq t) - \1(x_{m(n)} \leq t). 
\end{align*}
Consider now a fixed $n \in \mathcal B(N)$ with $n \in B_{i_n},$ then we know that $m(n) \in B_{i_n +\d_n},$ and we recall that $\d_n = \lfloor (1-\e)\log i_n\rfloor.$ Since $m(n) > N$ and $n \leq N,$ we know 
$$i_n > i_N-\d_N \eqcolon J_N,$$
else $i_n+\d_n \leq i_N.$ \\
This implies
\begin{equation}\label{Jthing}
    \mathcal B(N) \subseteq \bigcup_{J_N \leq i \leq i_N} B_i.
\end{equation}
 Notice that, by construction,
\begin{equation*}
    \|x_n-x_{m(n)}\| \leq \begin{cases}
            g(2^{i_n}), & \text{ if } \tn_k = x_n \text{ for some } k < n, \\
            g(2^{i_n-\d_n}), & \text{ else.}
    \end{cases}
\end{equation*}
The second term dominates, hence, upon recalling the definition of $g(\cdot)$ in (\ref{par}), 
\begin{equation*}
    \|x_n-x_{m(n)}\| \ll \frac{i_n^{1-\e}\log\log i_n}{2^{i_n}}.
\end{equation*}
Returning to (\ref{discplace}) we see that the summand
\begin{equation*}
    \1(x_n \leq t) - \1(x_{m(n)} \leq t)
\end{equation*}
contributes (is non-zero), whenever $t$ lies between $x_n$ and $x_{m(n)},$ so, together with (\ref{Jthing}), this gives
\begin{equation*}
    \sum_{n \in \mathcal B(N)} \1(x_n \leq t) - \1(x_{m(n)} \leq t) \leq \sum_{J_N \leq i_n \leq i_N} \1\left(\|x_n-t\| \leq \frac{i_n^{1-\e}\log\log i_n}{2^{i_n}}\right),
\end{equation*}
so
\begin{align*}
    \frac{1}{N}\sup_{t \in [0,1)}\left|\#\{1 \leq n \leq N \ | \ \nu_n \leq t\} - \#\{1 \leq n \leq N \ | \ x_n \leq t\}\right| \ll \frac{1}{N} \cdot(i_N - J_N) \cdot i_N^{1-\e} \log\log i_N
\end{align*}
and with $i_N - J \asymp \log i_N,$ we get in total
\begin{align*}
    D_N^*((\nu_n)_{n \leq N}) \ll \frac{\log N}{N} + \frac{(\log N)^{1-\e}\log\log N \log\log \log N}{N} = O\left(\frac{\log N}{N}\right).
\end{align*}
\end{proof}

\section{Pair correlation} 
This is the more involved part of the proof. Define
\begin{equation*}
    R_{s,N}((x_n)) = \frac{1}{N}\#\left\{1 \leq n \neq m \leq N \ | \ \|x_n-x_m\| \leq \frac{s}{N}\right\}.
\end{equation*}
The strategy to obtain Poissonian pair correlation, that is $R_{s,N}((\nu_n)) \rightarrow 2s,$ is to reduce from $(\nu_n)$ to the sequence $(\ts_n)$ in two steps. For the continuously perturbed sequence $(\ts_n)$, the almost sure Poissonian pair correlation is shown in Proposition \ref{ppccont}.

\subsection{Reduction to the continuously perturbed sequence.} Recall that $(\tn_n)$ is the discretized perturbed sequence with no point recycling. To avoid dealing with the restoration throughout, we first show that the pair correlation of $(\nu_n)$ is asymptotically the same as the pair correlation of $(\tn_n).$ This comes from the fact that the recycling procedure changes only relatively few indices. 

We repeatedly use the following argument: Assume that $\|\tn_n-t\| \leq \frac{s}{N}$ for some $t \in [0,1).$ Then
\begin{equation*}
    \|x_n-t\| \leq \|x_n-\tn_n\| + \|\tn_n - t\| \leq \frac{\log\log i_n}{2^{i_n}}+\frac{s}{N}.
\end{equation*}
Thus, we can bound
\begin{equation*}
    \#\left\{n \in B_i \ | \ \|\tn_n - t \| \leq \frac{s}{N}\right\} \ll \# \left\{n \in B_i \ | \ \|x_n-t\| \leq \frac{s}{N}+\frac{\log\log i}{2^i}\right\}.
\end{equation*}
Note that the van der Corput points $x_n$ with $n < 2^{i}$ have spacing $\frac{1}{2^{i}},$ hence we obtain
\begin{equation}\label{tri}
    \sup_{t \in [0,1)}\#\left\{n \in B_i \ | \ \|\tn_n-t\| \leq \frac{s}{N}\right\} \ll s2^{i-i_N}+\log\log i.
\end{equation}
The same estimate holds verbatim for $(\ts_n).$ 
Similarly, if $\|\nu_n -t \| \leq \frac{s}{N}$ for some $t \in [0,1),$ then
\begin{equation*}
    \|x_n-t\| \leq \|x_n-\nu_n\|+\|\nu_n-t\| \leq \frac{i_n^{1-\e}\log\log i_n}{2^{i_n}} + \frac{s}{N}.
\end{equation*}
We then obtain
\begin{equation}\label{tri2}
    \sup_{t \in [0,1)}\#\left\{n \in B_i \ | \ \|\nu_n-t\| \leq \frac{s}{N}\right\} \ll s2^{i-i_N}+i^{1-\e}\log\log i.
\end{equation}

\medskip

We show now that the point restoration does not change the pair correlation. To this end, put
\begin{equation*}
    A = \{1 \leq n \leq N \ | \ \nu_n \neq \tn_n\}.
\end{equation*}
Notice that
\begin{equation*}
    \# A \ll \frac{2^{i_N}}{i_N^{1-\e}}.
\end{equation*}
This is by construction: In $A,$ we collect points that are used for the recycling process rather than ordinarily perturbed. In each dyadic block $B_i,$ these points must come from the previous block $B_\ell$ with $\ell = i -\lfloor (1-\e)\log i\rfloor,$ so there cannot be more than $\#B_\ell \asymp \frac{2^i}{i^{1-\e}}$ such points. Write then
\begin{equation*}
    \#A = \# \{1 \leq n \leq N \ | \nu_n \neq \tn_n\} = \sum_{i \leq i_N} \#\{n \in B_i \ | \ \nu_n \neq \tn_n\} = \sum_{i \leq i_N} \frac{2^i}{i^{1-\e}}.
\end{equation*}
Note that
\begin{equation*}
    \frac{2^{i+1}}{(i+1)^{1-\e}}\cdot \frac{i^{1-\e}}{2^i} \geq 2^\e > 1,
\end{equation*}
so
\begin{equation*}
    \sum_{i \leq i_N} \frac{2^{i}}{i^{1-\e}} = \sum_{k \leq i_N-1} \frac{2^{i_N-k}}{(i_N-k)^{1-\e}} \leq \frac{2^{i_N}}{i_N^{1-\e}}\sum_{k \leq i_N-1} 2^{-\e k} \ll \frac{2^{i_N}}{i_N^{1-\e}}.
\end{equation*}

\begin{proposition} \label{resto}
    The sequences $(\nu_n)$ and $(\tn_n)$ asymptotically have the same pair correlation.
\end{proposition}

\begin{proof} 
    It suffices to show
    \begin{equation}\label{claim4.1}
        \#\left\{1 \leq n \neq m \leq N, \ n \in A \ | \ \|\tn_n-\tn_m\| \leq \frac{s}{N}\right\} = o(N)
    \end{equation}
    and
    \begin{equation}\label{claim4.1.2}
        \#\left\{1 \leq n \neq m \leq N, \ n \in A \ | \ \|\nu_n-\nu_m\| \leq \frac{s}{N}\right\} = o(N).
    \end{equation}
    Indeed, write
    \begin{align*}
         \#\left\{1 \leq n \neq m \leq N \ | \ \|\nu_n-\nu_m\| \leq \frac{s}{N}\right\} = \#&\left\{1 \leq n \neq m \leq N, \ n,m \notin A \ | \ \|\nu_n-\nu_m\| \leq \frac{s}{N}\right\} \\
        &+  \#\left\{1 \leq n\neq m \leq N, \ n \in A \ | \ \|\nu_n-\nu_m\| \leq \frac{s}{N}\right\}. 
    \end{align*}
    Noting that, if $n,m \notin A,$ $\nu_n = \tn_n$ and $\nu_m = \tn_m$ and assuming (\ref{claim4.1}) and (\ref{claim4.1.2}), this gives 
    \begin{align*}
        \frac{1}{N}& \#\left\{1 \leq n \neq m \leq N \ | \ \|\nu_n-\nu_m\| \leq \frac{s}{N}\right\} \\
        &= \frac{1}{N}\#\left\{1 \leq n \neq m \leq N, \ n,m \notin A \ | \ \|\tn_n-\tn_m\| \leq \frac{s}{N}\right\} + o(1) \\
        &= \frac{1}{N}\#\left\{1 \leq n \neq m \leq N \ | \ \|\tn_n-\tn_m\| \leq \frac{s}{N}\right\} + o(1),
    \end{align*}
    which is the claim. \\
    
    In order to show (\ref{claim4.1}), we treat big and small indices separately. To this end, put
    \begin{align*}
        I_1 = \{1 \leq n \leq 2^{\lfloor i_N- i_N^{1-2\e}\rfloor}\}, \quad I_2 = \{2^{\lfloor i_N- i_N^{1-2\e}\rfloor} < n \leq N\}
    \end{align*}
    and 
    \begin{equation*}
        A_1 = A \cap I_1, \quad A_2 = A \cap I_2.
    \end{equation*}
    Notice that, by (\ref{tri}), resp. (\ref{tri2}), we have
    \begin{align*}
        &\sup_{t \in [0,1)} \#\left\{n \in I_1 \ | \ \|\tn_n-t\| \leq \frac{s}{N}\right\} \ll (i_N-i_N^{1-2\e})\log\log(i_N-i_N^{1-2\e}), \\
        &\sup_{t \in [0,1)} \#\left\{n \in I_2 \ | \ \|\tn_n - t \| \leq \frac{s}{N}\right\} \ll i_N^{1-2\e} \log\log i_N, \\\
        &\sup_{t \in [0,1)} \#\left\{1 \leq n \leq N \ | \ \|\nu_n-t\| \leq \frac{s}{N}\right\} \ll i_N^{1-\e} \log\log i_N. 
    \end{align*}
    and
    \begin{align*}
        \# A_1 \ll \frac{2^{i_N-i_N^{1-2\e}}}{i_N^{1-\e}}, \quad \#A_2 \leq \#A \ll \frac{2^{i_N}}{i_N^{1-\e}}.
    \end{align*}
    For $n \in I_1,$ we have
    \begin{align*}
        \#A_1& \cdot \sup_{t \in [0,1)}\#\left\{n \in I_1 \ | \ \|\tn_n-t\| \leq \frac{s}{N}\right\}  \\
        &\ll \frac{2^{i_N-i_N^{1-2\e}}}{i_N^{1-\e}}(i_N-i_N^{1-2\e})\log\log (i_N-i_N^{1-2\e}) = o(N)
    \end{align*}
    and
    \begin{align*}
        \# \left\{n \in I_1, \ m \in A_2 \ | \ \|\tn_n-\tn_m\| \leq \frac{s}{N}\right\} \leq \#I_1 \cdot \sup_{t \in [0,1)} \left\{n \in A_2 \ | \ \|\tn_n - t \| \leq \frac{s}{N}\right\} \\
        \ll 2^{i_N- i_N^{1-2\e}}i_N^{1-2\e}\log\log i_N = o(N).
    \end{align*}
    Similarly, for $n \in I_2,$ we have
    \begin{equation*}
        \#A_1 \cdot \sup_{t \in [0,1)} \#\left\{n \in I_2 \ | \ \|\tn_n - t \| \leq \frac{s}{N}\right\} \ll \frac{2^{i_N-i_N^{1-2\e}}}{i_N^{1-\e}}i_N^{1-2\e}\log\log i_N = o(N)
    \end{equation*}
    and \begin{align*}    
        \#A_2  \cdot \sup_{t \in [0,1)}\#\left\{n \in I_2 \ | \ \|\tn_n-t\| \leq \frac{s}{N}\right\} &\ll \#A \cdot \sup_{t}\#\left\{n \in I_2 \ | \ \|\tn_n-t\| \leq \frac{s}{N}\right\} \\
        & \ll \frac{2^{i_N}}{i_N^{1-\e}}\cdot i_N^{1-2\e}\log\log i_N = o(N).
    \end{align*} 
    
    Turning to (\ref{claim4.1.2}), we notice that early restored points are harmless
    \begin{equation*}
        \#A_1 \cdot \sup_{t \in [0,1)}\#\left\{1 \leq n \leq N \ | \|\nu_n-t\| \leq \frac{s}{N}\right\} \ll \frac{2^{i_N}}{2^{i_N^{1-2\e}}i_N^{1-\e}} \cdot i_N^{1-\e} \log\log i_N = o(N).  
    \end{equation*}
    For $n \in A_2$ and $m \in I_1,$ we write
    \begin{equation*}
        \#\left\{n \in A_2, \ m \in I_1 \ | \ \|\nu_n-\nu_m\| \leq  \frac{s}{N}\right\} \leq \#I_1 \cdot \sup_{t \in [0,1)}\left\{n \in A_2 \ | \ \|\nu_n-t\| \leq \frac{s}{N}\right\}.
    \end{equation*}
    We now require a better bound than the general (\ref{tri2}). If $\nu_n$ is a recycled point, i.e. $n \in A,$ then it comes from a point in $B_{i_n-\lfloor (1-\e)\log i_n\rfloor}.$ Put $\ell_n = i_n-\lfloor (1-\e)\log i_n\rfloor$. Since $i_n \in (i_N-\lfloor i_N^{1-2\e}\rfloor, i_N],$ an interval of length $\lfloor i_N^{1-2\e}\rfloor,$ the $\ell_n$ also lie in an interval of length $\lfloor i_N^{1-2\e}\rfloor,$ and $\ell_n \leq i_N-(1-\e)\log i_N.$ Inside a block $B_{\ell_n},$ the van der Corput points have spacing
    \begin{equation*}
        2^{-\ell_n} \gg \frac{i_N^{1-\e}}{2^{i_N}},
    \end{equation*}
    hence an $\frac{s}{N}-$interval contains $\ll_s 1$ such points. Summing this yields
    \begin{equation*}
        \sup_{t \in [0,1)}\#\left\{n \in A_2 \ | \ \|\nu_n-t\| \leq \frac{s}{N}\right\} \ll i_N^{1-2\e}.
    \end{equation*}
    Thus
    \begin{equation*}
        \#I_1 \cdot \sup_{t \in [0,1)}\left\{n \in A_2 \ | \ \|\nu_n-t\| \leq \frac{s}{N}\right\} \ll 2^{i_N-i_N^{1-2\e}}i_N^{1-2\e} = o(N).
    \end{equation*}
    If $m \in I_2,$ we split this into $m \in A_2$ and $m \notin A_2.$ If $m \notin A_2,$ $\nu_m = \tn_m$ and we are done. Else,
    \begin{align*}
        \#&\left\{n,m \in A_2 \ | \ \|\nu_n-\nu_m\| \leq \frac{s}{N}\right\} \\
        &\leq \#A_2 \cdot \sup_{t \in [0,1)} \#\left\{n \in A_2 \ | \ \|\nu_n-t\| \leq \frac{s}{N}\right\} \ll \frac{2^{i_N}}{i_N^{1-\e}}i_n^{1-2\e} = o(N).
    \end{align*}
    This concludes the proof. 
\end{proof}

We now show that we can further reduce to the continuously perturbed sequence $(\ts_n).$ This follows roughly from the fact that the distance between $\tn_n$ an $\ts_n$ for large $n$ becomes very small, and that there are relatively few points for small $n.$

\begin{proposition}\label{samepc}
    If $(\ts_n)$ has Poissonian pair correlation, then so does $(\tn_n).$
\end{proposition}

\begin{proof}
    Let $N$ be given. We split the indices into small, medium and large as follows
    \begin{equation*}
        I_1 = \left\{1 \leq n \leq 2^k\right\}, \quad I_2 = \left\{2^k < n \leq 2^\ell\right\}, \quad I_3 = \left\{2^\ell < n \leq N\right\}
    \end{equation*}
    for the cutoff parameters
    \begin{equation*}
        \ell = {j} + \log\log i_N, \quad k = i_N - i_N^{1-2\e}
    \end{equation*} 
    where, as before,
    \begin{equation*}
        j = i_N-\lfloor(1-\e)\log i_N\rfloor.
    \end{equation*}
    We record that
    \begin{equation*}
        \#I_1 \asymp \frac{N}{2^{\log^{1-2\e}N}}, \quad \#I_2 \asymp \frac{N \log\log N}{(\log N)^{1-\e}}, \quad \#I_3 \asymp N.
    \end{equation*}
    Write
    \begin{align*}
        \frac{1}{N} &\#\left\{ 1 \leq n \neq m \leq N \ | \ \|\tn_n-\tn_m\| \leq \frac{s}{N}\right\} \\
        &= \frac{1}{N}\left( \sum_{\substack{ n \neq m \\ n,m \in I_1}} + \sum_{\substack{ n \neq m \\ n,m \in I_2}} + \sum_{\substack{ n \neq m \\ n,m \in I_3}} + 2\sum_{\substack{ n \neq m \\ n \in I_1, \ m \in I_2 \cup I_3}} + 2 \sum_{\substack{ n \neq m \\ n \in I_2, \ m \in I_3}} \right) \1\left(\|\tn_n-\tn_m\| \leq \frac{s}{N}\right),
    \end{align*}
    and similarly
    \begin{align*}
        \frac{1}{N}&\#\left\{1 \leq n \neq m \leq N \ | \ \|\ts_n-\ts_n\| \leq \frac{s}{N}\right\} \\
        &= \frac{1}{N}\left( \sum_{\substack{ n \neq m \\ n,m \in I_1}} + \sum_{\substack{ n \neq m \\ n,m \in I_2}} + \sum_{\substack{ n \neq m \\ n,m \in I_3}} + 2\sum_{\substack{ n \neq m \\ n \in I_1, \ m \in I_2 \cup I_3}} + 2 \sum_{\substack{ n \neq m \\ n \in I_2, \ m \in I_3}} \right) \1\left(\|\ts_n-\ts_m\| \leq \frac{s}{N}\right). 
    \end{align*}
    It suffices to show 
    \begin{enumerate}[label= (\roman*)]
        \item $$\frac{1}{N}\sum_{\substack{n \neq m \\ n,m \in I_3}}\1\left(\|\tn_n-\tn_m\| \leq \frac{s}{N}\right) = \frac{1}{N}\sum_{\substack{n \neq m  \\ n,m \in I_3}}\1\left(\|\ts_n-\ts_m\| \leq \frac{s}{N}\right) + o(1),$$ \\
        \item $$\sum_{\substack{n \neq m \\ n\in I_1, \ 1 \leq m \leq N }}\1\left(\|\tn_n-\tn_m\| \leq \frac{s}{N}\right) = o(N),$$ \\
        \item $$\sum_{\substack{n \neq m  \\ n,m \in I_2}}\1\left(\|\tn_n-\tn_m\| \leq \frac{s}{N}\right) = o(N), \quad \text{and} $$ \\
        \item $$\sum_{\substack{n \neq m \\ n \in I_2, \ m \in I_3}}\1\left(\|\tn_n-\tn_m\| \leq \frac{s}{N}\right) = o(N).$$
    \end{enumerate}
    Noting that these also hold for $(\ts_n),$ we indeed get
    \begin{align*}
        \frac{1}{N} &\#\left\{ 1 \leq n \neq m \leq N \ | \ \|\tn_n-\tn_m\| \leq \frac{s}{N}\right\} \\
        &= \frac{1}{N}\left( \sum_{\substack{ n \neq m \\ n,m \in I_1}} + \sum_{\substack{ n \neq m \\ n,m \in I_2}} + \sum_{\substack{ n \neq m \\ n,m \in I_3}} + 2\sum_{\substack{ n \neq m \\ n \in I_1, \ m \in I_2 \cup I_3}} + 2 \sum_{\substack{ n \neq m \\ n \in I_2, \ m \in I_3}} \right) \1\left(\|\tn_n-\tn_m\| \leq \frac{s}{N}\right) \\
        &= \frac{1}{N} \sum_{\substack{ n \neq m \\ n,m \in I_3}} \1\left(\|\tn_n-\tn_m\| \leq \frac{s}{N}\right) + o_N(1) = \frac{1}{N} \sum_{\substack{n \neq m \\ n,m \in I_3}} \1\left(\|\ts_n-\ts_m\| \leq \frac{s}{N}\right) + o_N(1) \\
        & = \frac{1}{N} \#\left\{1 \leq n \neq m \leq N \ | \ \|\ts_n-\ts_m\| \leq \frac{s}{N}\right\} + o_N(1).
    \end{align*}

    \textit{ad (i).} We first notice that, for large indices $n \in I_3$, 
    \begin{equation*}
      \sup_{i_n \geq \ell} \ |\ts_n-\tn_n| \ll \frac{1}{2^\ell \ell^{1-\e}}
    \end{equation*}
    due to spacing of points in $\mathcal C_n$ and $\tn_n = \argmin_{\nu \in \mathcal C_n} |\nu - \ts_n|.$ \\ 
    Hence, if $i_n \geq \ell,$
    \begin{equation*}
        \frac{1}{2^{i_n}i_n^{1-\e}} \ll \frac{1}{N \log \log N} = o\left(\frac{1}{N}\right),
    \end{equation*}
    so we obtain 
    \begin{align*}
        \frac{1}{N}&\#\left\{2^\ell \leq n \neq m \leq N \ | \ \|\tn_n-\tn_m\| \leq \frac{s}{N}\right\} \\
        &= \frac{1}{N}\#\left\{2^\ell \leq n \neq m \leq N \ | \ \|\ts_n-\ts_m\| \leq \frac{s}{N}\right\} + o(1)
    \end{align*}
    as $N \rightarrow \infty.$ \\
    As in the proof of Proposition \ref{resto}, we use the estimate
    \begin{equation*}
        \#\left\{n \in I_i, \ m \in I_j \ | \ \|\tn_n-\tn_m\| \leq \frac{s}{N}\right\} \ll \#I_i \cdot \sup_{t \in [0,1)}\#\left\{m \in I_j \ | \ \|\tn_m-t\| \leq \frac{s}{N}\right\}
    \end{equation*}
    for $i,j \in \{1,2,3\}.$
    \textit{ad (ii).} Using that $\# I_1$ is small, we get
    \begin{align*}
        \#I_1 \cdot \sup_{t \in [0,1)}\#\left\{m \in I_j \ | \ \|\tn_m-t\| \leq \frac{s}{N}\right\} &\ll \begin{cases}
            \frac{N}{2^{\log^{1-2\e}N}} \cdot k \log\log k, \ j= 1, \\
            \frac{N}{2^{\log^{1-2\e}N}} \cdot i_N^{1-2\e}\log\log \ell, \ j = 2, \\
            \frac{N}{2^{\log^{1-2\e}N}} \cdot (i_N-\ell)\log\log i_N, \ j = 3
        \end{cases} \\
        &= o(N)
    \end{align*}
    \textit{ad (iii).} For $n,m \in I_2,$ 
    \begin{equation*}
        \#I_2 \cdot \sup_{t \in [0,1)}\#\left\{m \in I_2 \ | \ \|\tn_m-t\| \leq \frac{s}{N}\right\} \ll \frac{N \log \log N \log\log\log N}{\log ^\e N} = o(N).
    \end{equation*}
    \textit{ad (iv).} Finally,
    \begin{equation*}
        \sup_{t \in [0,1)}\#\left\{n \in I_3 \ | \ \|\tn_n-t\| \leq \frac{s}{N}\right\} \ll s + \log i_N \log\log i_N,
    \end{equation*}
    so
    \begin{equation*}
        \# I_2 \cdot \sup_{t \in [0,1)} \left\{n \in I_3 \ | \ \|\tn_n - t \| \leq \frac{s}{N}\right\} \ll N \left(\frac{s\log i_N + \log^2i_N \log\log i_N}{i_N^{1-\e}} \right) = o(N).
    \end{equation*}
    Note that, since the same replacement bounds apply for $(\ts_n),$ these estimates also hold when replacing $\tn_n$ with $\ts_n.$ \\
    This yields the claim.
\end{proof}

\subsection{Pair correlation of the continuous sequence.} This argument follows \cite{lachman2021additiveenergydiscrepancypoissonian} loosely. Define
\begin{equation*}
    R_{s,N} \coloneq R_{s,N}((\ts_n)_{n \leq N}) = \frac{1}{N} \sum_{1 \leq n \neq m \leq N} \1\left(\|\ts_n-\ts_m\| \leq \frac{s}{N}\right).
\end{equation*}
This is a random variable with $\ts_n = \|x_n + y_n \|$ and $y_n \sim \text{Unif}([-g(n),g(n)]).$
\begin{proposition}\label{ppccont}
    The sequence $(\ts_n)$ has Poissonian pair correlation almost surely.
\end{proposition}

\begin{proof}
We want to show
\begin{equation*}
    R_{s,N} \longrightarrow 2s
\end{equation*}
almost surely. 
We proceed in four steps: \begin{enumerate}
    \item $\E\left[R_{s,N}\right] \rightarrow 2s.$
    \item Establish and upper bound for $\Var(R_{s,N})$.
    \item Apply Chebychev's inequality.
    \item Apply Borel-Cantelli lemma.
\end{enumerate}

Define for each $n$
\begin{equation*}
    f_n(x) = \frac{1}{2g(n)}\1(\|x_n-x\| \leq g(n)).
\end{equation*}
This is the density of $\ts_n$: For any $A \in \mathcal B([0,1))$
\begin{align*}
    \P(\ts_n \in A) = \int_0^1 \1(\ts_n \in A) d\lambda = \frac{1}{2g(n)}\int_{-g(n)}^{g(n)} \1(\|x_n+y_n\| \in A) dy_n.
\end{align*}
Changing variables $\|x_n+y_n\| \mapsto x$ with support condition $\1(\|x_n-x\| \leq g(n))$ gives
\begin{align*}
    \frac{1}{2g(n)}&\int_{-g(n)}^{g(n)} \1(\|x_n+y_n\| \in A) dy_n \\
    &= \int_0^1 \1(x \in A) \frac{1}{2g(n)}\1(\|x_n-x\| \leq g(n)) dx = \int_0^1 \1(x \in A) f_n(x)dx.
\end{align*}

\medskip

\textbf{Step 1.} We want to show
\begin{equation}\label{step1e}
    \E[R_{s,N}] = \frac{1}{N}\sum_{1 \leq n \neq m \leq N} \E\left[\1\left(\|\ts_n-\ts_m\| \leq \frac{s}{N}\right)\right] \longrightarrow 2s, \quad N \rightarrow \infty.
\end{equation}
We have
\begin{equation*}
    \E\left[\1\left(\|\ts_n-\ts_m\| \leq \frac{s}{N}\right)\right] = \int_0^1\int_0^1 \1\left(\|x-y\|\right)f_n(x)f_n(y)dxdy,
\end{equation*}
so
\begin{equation*}
    \E[R_{s,N}] = \frac{1}{N} \int_0^1\int_0^1 \1\left(\|x-y\| \leq \frac{s}{N}\right)\sum_{1 \leq n \neq m \leq N}f_n(x)f_m(y)dxdy.
\end{equation*}
We introduce the diagonal $x = y$ via
\begin{equation*}
    \sum_{1 \leq n \neq m \leq N}f_n(x)f_m(y) = \sum_{1 \leq n \leq N}\sum_{1 \leq m \leq N}f_n(x)f_n(y) - \sum_{n \leq N}f_n(x)f_n(y).
\end{equation*}
Put $F_N(x) = \sum_{1 \leq n \leq N}f_n(x).$ In order to show (\ref{step1e}), we write
\begin{align*}
    \E[R_{s,N}] = \frac{1}{N} \int_0^1\int_0^1 &\1\left(\|x-y\| \leq \frac{s}{N}\right)F_N(x)F_N(y)dxdy \\
    &- \frac{1}{N} \sum_{1 \leq n \leq N} \int_0^1\int_0^1 \1\left(\|x-y\| \leq \frac{s}{N}\right)f_n(x)f_n(y)dxdy
\end{align*}
and proceed as follows: We show that the diagonal is negligible, that is,
\begin{equation}\label{diag}
    \frac{1}{N}\sum_{1 \leq n \leq N}\int_0^1\int_0^1 \1\left(\|x-y\| \leq \frac{s}{N}\right)f_n(x)f_n(y)dxdy = o_N(1),
\end{equation}
and that we get the main term from
\begin{equation}\label{expmt}
    \frac{1}{N} \int_0^1\int_0^1 \1\left(\|x-y\| \leq \frac{s}{N}\right)F_N(x)F_N(y)dxdy = 2s (1+o_N(1)).
\end{equation}
The latter follows from: \\
\textit{Claim.}
\begin{equation}\label{claimex}
    F_N(\cdot) = N(1+o(1)).
\end{equation}
     
We defer the proof of the claim for now. Together, (\ref{diag}) and (\ref{expmt}) yield
\begin{align*}
    \E[R_{s,N}] =& \frac{1}{N}\int_0^1 \int_0^1 \1\left(\|x-y\| \leq \frac{s}{N}\right)F_N(x)F_N(y)dxdy \\
    &- \frac{1}{N} \sum_{1 \leq n \leq N} \int_0^1\int_0^1 \1\left(\|x-y\| \leq \frac{s}{N}\right)f_n(x)f_n(y)dxdy \\
    =& N(1+o_N(1)) \int_0^1\int_0^1 \1\left(\|x-y\| \leq \frac{s}{N}\right)dxdy - o_N(1) \\
    =& N\frac{2s}{N} (1+o_N(1)) = 2s (1+o_N(1)),
\end{align*}
which gives the claim. \\
We turn to proving (\ref{diag}). Notice that $f_n(x) \leq \frac{1}{2g(n)},$ so
\begin{equation*}
    \int_0^1 \1\left(\|x-y\| \leq \frac{s}{N}\right)f_n(x)dx \leq \frac{1}{2g(n)} \int_0^1 \1\left(\|x-y\| \leq \frac{s}{N}\right)dx.
\end{equation*}
Plugging this in, we obtain
\begin{align*}
    \frac{1}{N}\sum_{1 \leq n \leq N}\int_0^1\int_0^1 \1\left(\|x-y\|\leq \frac{s}{N}\right)f_n(x)f_n(y)dxdy\\
    \leq \frac{1}{N} \sum_{1 \leq n \leq N} \frac{s}{Ng(n)}\int_0^1f_n(y)dy = o_N(1).
\end{align*}

\medskip

\textit{Proof of Claim.} It remains to show (\ref{claimex}).
In proving this, a discrepancy bound arises naturally, and we want to utilize our good understanding of the discrepancy of the van der Corput sequence $x_n$ for this. Doing this immediately over the full-length interval, however, introduces a non-negligible factor $g(N)^{-1},$ which is too large, as $g(n)$ decreases with increasing $n.$ So we want to use a subsequence $(N_k)_{k \geq 1}$ to division $[1,N]$ into smaller intervals $[N_k,N_{k+1})$ such that $D_{N_k}((x_n)_{n \leq N_k})$ is particularly small, similar to powers of $2,$ but we also require $\frac{N_{k+1}}{N_k} \rightarrow 1,$ so that we can use a sandwiching argument.  \\
We construct such a sequence to emulate powers of $2,$ whilst retaining the growth condition, which leads us to consider $N_k$ with sparse and large dyadic digits as follows. \\
Write 
\begin{equation*}
    N = \sum_{i \geq 0}\e_i(N)2^i, \quad \e_i(N) \in \{0,1\},
\end{equation*}
and denote by $I_N = \{i \ | \e_i(N) = 1\}$ and $m(N) = \min I_N$ the minimal non-zero digit index. Put 
\begin{equation}
    L(\ell) = \lceil \log\log\log \ell \rceil
\end{equation}
and define, for each $\ell \geq 1,$
\begin{equation*}
    S_\ell = \left\{2^\ell \leq N < 2^{\ell+1} \ | \ \e_i(N) = 0 \text{ for all } i \leq \ell - L(\ell)\right\}. 
\end{equation*}
Finally, let $(N_k)_{k \geq 1}$ be the sequence of elements in $\bigcup_{\ell \geq 1}S_\ell.$ One can easily see that this satisfies, for $N_k \in S_r$:
\begin{enumerate}[label =(\roman*)]
    \item Only the last $L(r)$ digits can be non-zero, i.e.
        $$|I_{N_k}| \leq L(r).$$
    \item The minimal non-zero digit index is ``large'', i.e.
        $$m(N_k) \geq r-L(r) \rightarrow \infty.$$
    \item $\frac{N_{k+1}}{N_k} \rightarrow 1.$
\end{enumerate}

Assume for now that $N = N_{K}$ for some $K.$ We decompose then
\begin{equation*}
    F_{N_K}(x) = \sum_{1 \leq n \leq N_K} f_n(x) = \sum_{1 \leq k \leq K-1} \sum_{N_k < n \leq N_{k+1}} f_n(x).
\end{equation*}
Due to condition (iii), $g(\cdot)$ is asymptotically constant on such a block, so, plugging in the definition of $f_n(x),$
\begin{align*}
    \sum_{N_k \leq n < N_{k+1}}& \frac{1}{2g(n)} \1(x_n \in [x-g(n),x+g(n)]) \\
    &= \frac{1}{2g(N_k)}(1+o_N(1)) \sum_{N_k \leq n < N_{k+1}} \1(x_n \in [x-g(n),x+g(n)]).
\end{align*}
By monotonicity of $g(\cdot),$ we can bound this from above and below by
\begin{equation*}
    \left(\frac{1}{2g(N_k)} \sum_{N_k \leq n < N_{k+1}} \1(x_n \in [x-c,x+c])\right)(1+o_N(1))
\end{equation*}
for $c \in \{g(N_k),g(N_{k+1})\}.$ Write
\begin{align*}
   \frac{1}{2g(N_k)}& \sum_{N_k < n \leq N_{k+1}} \1(x_n \in [x-c,x+c]) =\\
    &\frac{1}{2g(N_k)} \left(\sum_{1 \leq n \leq N_{k+1}} \1(x_n \in [x-c,x+c]) - \sum_{1 \leq n \leq N_k} \1(x_n \in [x-c,x+c])\right).
\end{align*} 
By definition, 
\begin{equation*}
    \sum_{1 \leq n \leq N_j}\1\left(x_n \in [x-c,x+c]\right) = 2cN_j+ O(N_jD_{N_j})
\end{equation*}
with $j \in \{k,k+1\}.$ If $N_j \in S_{\ell_j},$
\begin{equation*}
    D_{N_j}((x_n)_{n \leq N_j}) \ll \frac{L(\ell_j)}{N_j}
\end{equation*}
due to the sparsity of digits $N_j.$ Therefore,
\begin{equation*}
    \frac{1}{2g(N_k)} \sum_{N_k \leq n < N_{k+1}} \1(x_n \in [x-c,x+c]) = (N_{k+1}-N_k)(1+o(1)),
\end{equation*}
which gives
\begin{equation*}
    \sum_{1 \leq k \leq K-1}\sum_{N_k \leq n < N_{k+1}}f_n(x) = \left(\sum_{1 \leq k \leq K-1} N_{k+1}-N_k\right)(1+o(1)) = N_K + o(1).
\end{equation*}
For a general $N$ with $N_k \leq N < N_{k+1},$ since $f_n(x) \geq 0,$ we have
\begin{equation*}
    F_{N_k} (x) \leq F_N(x) \leq F_{N_{k+1}}(x),
\end{equation*}
and with $\frac{N_{k+1}}{N_k} \rightarrow 1,$ this yields
\begin{equation*}
    F_N(x) = N(1+o_N(1))
\end{equation*}
for all $N,$ which proves (\ref{claimex}), and thus concludes Step 1.

\textbf{Step 2.} We now turn to bounding
$$\Var(R_{s,N}) = \E[R_{s,N}^2]-\E[R_{s,N}]^2.$$
Compute
\begin{align*}
   \E[R_{s,N}^2] = \frac{1}{N^2}& \sum_{\substack{1 \leq n_1,n_2,m_1,m_2 \leq N \\ n_1 \neq m_1, \ n_2 \neq m_2}}\E\left[\1\left(\|\ts_{n_1}-\ts_{m_1}\| \leq \frac{s}{N}\right)\1\left(\|\ts_{n_2}-\ts_{m_2}\| \leq \frac{s}{N}\right)\right] \\
    = \frac{1}{N^2}& \left( \sum_{\substack{1 \leq n_1,n_2,m_1,m_2 \leq N \\ n_1,n_2,m_1,m_2 \text{ distinct}}} + \sum_{\substack{1 \leq n_1,n_2,m_1,m_2 \leq N \\ m_1 \neq n_1, \ m_2 \neq n_2 \\ \{n_1,m_1\} \cap \{n_2,m_2\} \neq \emptyset}}\right) \\
    &\E\left[\1\left(\|\ts_{n_1}-\ts_{m_1}\| \leq \frac{s}{N}\right)\1\left(\|\ts_{n_2}-\ts_{m_2}\| \leq \frac{s}{N}\right)\right].
\end{align*}
In the first sum $n_1,n_2,m_1,m_2$ are all distinct, thus the terms $y_{n_1}-y_{m_1}$ and $y_{n_2}-y_{m_2}$ are independent, so
\begin{align*}
    \frac{1}{N^2} \sum_{\substack{1 \leq n_1,n_2,m_1,m_2 \leq N \\ m_1,n_1, m_2,n_2 \ \text{distinct}}}\E\left[\1\left(\|\ts_{n_1}-\ts_{m_1}\| \leq \frac{s}{N}\right)\1\left(\|\ts_{n_2}-\ts_{m_2}\| \leq \frac{s}{N}\right)\right] \\
    = \frac{1}{N^2}\sum_{\substack{1 \leq n_1,n_2,m_1,m_2 \leq N \\ m_1,n_1, m_2,n_2 \ \text{distinct}}}\E\left[\1\left(\|\ts_{n_1}-\ts_{m_1}\| \leq \frac{s}{N}\right)\right] \E\left[\1\left(\|\ts_{n_2}-\ts_{m_2}\| \leq \frac{s}{N}\right)\right] \leq & \E[R_{s,N}]^2.
\end{align*}
Plugging this in yields
\begin{align*}
    \Var(R_{s,N})& = \E[R_{s,N}^2]-\E[R_{s,N}]^2 \\
    &= \frac{1}{N^2} \left( \sum_{\substack{1 \leq n_1,n_2,m_1,m_2 \leq N \\ n_1,n_2,m_1,m_2 \text{ distinct}}} + \sum_{\substack{1 \leq n_1,n_2,m_1,m_2 \leq N \\ m_1 \neq n_1, \ m_2 \neq n_2 \\ \{n_1,m_1\} \cap \{n_2,m_2\} \neq \emptyset}}\right) \\
    & \qquad\E\left[\1\left(\|\ts_{n_1}-\ts_{m_1}\| \leq \frac{s}{N}\right)\1\left(\|\ts_{n_2}-\ts_{m_2}\| \leq \frac{s}{N}\right)\right] -\E[R_{s,N}]^2 \\
    &\leq  \sideset{}{'}\sum_{1 \leq n_1,n_2,m_1,m_2 \leq N} \E\left[\1\left(\|\ts_{n_1}-\ts_{m_1}\|\leq \frac{s}{N}\right)\1\left(\|\ts_{n_2}-\ts_{m_2}\|\leq \frac{s}{N}\right)\right]\\
    &\qquad+ \E[R_{s,N}]^2- \E[R_{s,N}]^2
\end{align*}

with
$$\sideset{}{'}\sum_{1 \leq n_1,n_2,m_1,m_2 \leq N} = \sum_{\substack{1 \leq n_1,n_2,m_1,m_2 \leq N \\ m_1 \neq n_1, \ m_2 \neq n_2 \\ \{n_1,m_1\} \cap \{n_2,m_2\} \neq \emptyset}},$$
so
\begin{equation*}
    \Var(R_{s,N}) \leq \frac{1}{N^2}\sideset{}{'}\sum_{1 \leq n_1,n_2,m_1,m_2 \leq N} \E\left[\1\left(\|\ts_{n_1}-\ts_{m_1}\| \leq \frac{s}{N}\right)\1\left(\|\ts_{n_2}-\ts_{m_2}\| \leq \frac{s}{N}\right)\right].
\end{equation*}
We split this into 
$$\frac{1}{N^2}\left(\sum_{\substack{1 \leq n_1,n_2,m_1,m_2 \leq N \\ m_1 \neq n_1, \ m_2 \neq n_2 \\ |\{n_1,m_1\} \cap \{n_2,m_2\}| = 2}}+\sum_{\substack{1 \leq n_1,n_2,m_1,m_2 \leq N \\ m_1 \neq n_1, \ m_2 \neq n_2 \\ |\{n_1,m_1\} \cap \{n_2,m_2\}| = 1}}\right)\E\left[\1\left(\|\ts_{n_1}-\ts_{m_1}\| \leq \frac{s}{N}\right)\1\left(\|\ts_{n_2}-\ts_{m_2}\| \leq \frac{s}{N}\right)\right].$$

The first sum simply contributes
$$\frac{2}{N^2}\sum_{1 \leq n \neq m \leq N} \E\left[\1\left(\|\ts_n-\ts_m\| \leq \frac{s}{N}\right)\right] = \frac{2\E[R_{s,N}]}{N}.$$

The second sum contributes
$$\frac{4}{N^2}\sum_{\substack{1 \leq n_1,n_2,m \leq N \\ m \neq n_1 \neq n_2 \neq m}}\E\left[\1\left(\|\ts_{n_1}-\ts_m\| \leq \frac{s}{N}\right)\1\left(\|\ts_{n_2}-\ts_m\| \leq \frac{s}{N}\right)\right]$$
where
\begin{align*}
    \E&\left[\1\left(\|\ts_{n_1}-\ts_m\| \leq \frac{s}{N}\right)\1\left(\|\ts_{n_2}-\ts_m\| \leq \frac{s}{N}\right)\right] \\
    &= \int_0^1\int_0^1\int_0^1 \1\left(\|x-z\| \leq \frac{s}{N}\right)\1\left(\|y-z\| \leq \frac{s}{N}\right)f_{n_1}(x)f_{n_2}(y)f_m(z)dxdydz.
\end{align*}
Recalling
\begin{equation*}
    F_N(x) = \sum_{1 \leq n \leq N}f_n(x), \quad f_n(x) = \frac{1}{2g(n)}\1(\|x_n-x\| \leq g(n)),
\end{equation*}
with $f_n(\cdot)$ being the density of $\ts_n,$
we write
\begin{align*}
    \sum_{1 \leq m \leq N}& \E[\phi(\ts_m)] \\
    &= N \int_0^1 \left(\int_0^1 \int_0^1 \1\left(\|x-z\| \leq \frac{s}{N}\right)\1\left(\|y-z\| \leq \frac{s}{N}\|\right) f_{n_1}(x)f_{n_2}(y)dxdy \right) F_N(z)dz
\end{align*}
with
\begin{align*}
    \phi(z) = \int_0^1\int_0^1 \1\left(\|x-z\| \leq \frac{s}{N}\right)\1\left(\|y-z\| \leq \frac{s}{N}\right) f_{n_1}(x)f_{n_2}(y)dxdy.
\end{align*}
Summing over the remaining indices (and recalling the summed densities function $F_N(\cdot)$), i.e.
\begin{align*}
    \sum_{1 \leq n_1 \leq N} \int_0^1 \int_0^1 \1\left(\|\ts_{n_1}-x\| \leq \frac{s}{N}\right)\1\left(\|\ts_{n_2}-x\| \leq \frac{s}{N}\|\right) d\ts_{n_1}d\ts_{n_2} \\
    = N \int_0^1\left(\int_0^1 \1\left(\|y-x\| \leq \frac{s}{N}\right)\1\left(\|\ts_{n_2}-x\| \leq \frac{s}{N}\right) d\ts_{n_2}\right)F_N(y) dy
\end{align*}
and
\begin{align*}
    \sum_{1 \leq n_2 \leq N} \int_0^1\1\left(\|y-x\| \leq \frac{s}{N}\right)\1\left(\|\ts_{n_2}-x\| \leq \frac{s}{N}\right)d\nu_{n_2} \\
    = N \int_0^1 \1\left(\|y-x\| \leq \frac{s}{N}\right)\1\left(\|z-x\| \leq \frac{s}{N}\right) F_N(z) dz,
\end{align*}
yields
\begin{align*}
    &\sum_{\substack{1 \leq n_1,n_2,m \leq N \\ m \neq n_1 \neq n_2 \neq m}}\E\left[\1\left(\|\ts_{n_1}-\ts_m\| \leq \frac{s}{N}\right)\1\left(\|\ts_{n_2}-\ts_m\| \leq \frac{s}{N}\right)\right] \\
    & \quad \leq N^3\int_0^1\int_0^1\int_0^1 \1\left(\|y-x\| \leq \frac{s}{N}\right)\1\left(\|z-x\| \leq \frac{s}{N}\right)F_N(x)F_N(y)F_N(z) dxdydz \\
    & \quad = N^3 \int_0^1\left(\int_0^1 \1\left(\|y-x\| \leq \frac{s}{N}\right) F_N(y) dy\right)^2 F_N(x) dx.
\end{align*}
Hence
\begin{align*}
    \frac{4}{N^2}\sum_{\substack{1 \leq n_1,n_2,m \leq N \\ m \neq n_1 \neq n_2 \neq m}}\E[\1_{s,N}(\ts_{n_1}-\ts_m)\1_{s,N}(\ts_{n_2}-\ts_m)]  \ll 4N\left(\frac{2s}{N}\right)^2 = \frac{16s^2}{N}.
\end{align*}

\textbf{Step 3 \& 4.} To conclude, we first use Chebyshev's inequality, yielding
\begin{equation*}
    \P\left(\left|R_{s,N} - \E[R_{s,N}]\right| \geq \frac{1}{N^{1/4}}\right) \ll \frac{1}{N^{1/2}}.
\end{equation*}
The claim follows by Borel-Cantelli. Since this is not summable, we first consider the pair correlation statistic along a subsequence $(R_{s,N_m})$ with $\frac{N_{m+1}}{N_m} \rightarrow 1.$ For $N_M = M^4,$ we get
\begin{equation*}
    R_{s,N_M} \longrightarrow 2s, \quad N_M \rightarrow \infty
\end{equation*}
almost surely. As in Lemma 3.1 in \cite{Rudnick_Technau_2020}, this suffices to show
\begin{equation*}
    R_{s,N} \longrightarrow 2s, \quad N \rightarrow \infty
\end{equation*}
in general. 
\end{proof}

Thus $(\ts_n)_{n \geq 1}$ has Poissonian pair correlation almost surely, and, by Propositions \ref{resto} and \ref{samepc}, so does $(\nu_n)_{n \geq 1}.$

\section*{Acknowledgements}

The author was supported by the Austrian Science Fund (FWF), project 10.55776/PAT5120424. The author is grateful to Christoph Aistleitner and Manuel Hauke-Treuer for helpful comments on an earlier version of this paper.

\printbibliography

\end{document}